\documentclass[11pt]{amsart}
\newcommand{\IN}{\mathbb N}
\newcommand{\IC}{\mathbb C}
\newcommand{\IH}{\mathbb H}
\newcommand{\la}{\langle}
\newcommand{\ra}{\rangle}
\newcommand{\w}{\omega}
\newcommand{\e}{\varepsilon}

\newtheorem{theorem}{Theorem}
\newtheorem{lemma}{Lemma}
\newtheorem{claim}{Claim}

\title[On 2-groups with a unique 2-element subgroup]{The structure of infinite 2-groups\\ with a unique 2-element subgroup}
\author{Taras Banakh}
\address{Department of Mathematics, Ivan Franko National University of Lviv, Ukraine, and\newline
Instytut Matematyki, Uniwersytet Humanistyczno-Przyrodniczy jana Kochanowskiego, Kielce, Poland}
\email{tbanakh@gmail.com}

\begin{document}
\begin{abstract}
We prove that each infinite 2-group $G$ with a unique 2-element subgroup is isomorphic either to the quasicyclic 2-group $C_{2^\infty}$ or to the infinite group of generalized quaternions $Q_{2^\infty}$. The latter group is generated by the set $C_{2^\infty}\cup Q_8$ in the algebra of quaternions $\IH$.
\end{abstract}
\maketitle

In this paper we describe the structure of 2-groups that contain a unique 2-element subgroup. For finite groups this was done in \cite[5.3.6]{Rob}: Each finite 2-group with a unique 2-element subgroup is either cyclic or is a group of generalized quaternions.

Let us recall that a group $G$ is called a {\em 2-group} if each element $x\in G$ has order $2^k$ for some $k\in\IN$. The {\em order} of an element $x$ is the smallest number $n\in\IN$ such that $x^n=1$ where $1$ denotes the neutral element of the group. By $\w$ we denote the set of non-negative integer numbers.
 
For $n\in\w$ denote by
$$C_{2^n}=\{z\in \IC:z^{2^n}=1\}$$ the cyclic group of order $2^n$. The union $$C_{2^\infty}=\bigcup_{n\in\IN}C_{2^n}\subset\IC$$is called the {\em quasicyclic 2-group}.

The {\em group of quaternions} is the 8-element subgroup
$$Q_8=\{1,-1,i,-i,j,-j,k,-k\}$$in the  algebra of quaternions $\IH$ (endowed with the operation of multiplication of quaternions). The real algebra $\IH$ contains the field of complex numbers $\IC$ as a subalgebra.

For $n\in\IN$ the subgroup $Q_{2^n}$ of $\IH$  generated by the set $C_{2^{n-1}}\cup Q_8$ is called the {\em group of generalized quaternions}.
For $n\ge 3$ this group has a presentatiom
$$\la x,y\mid x^4=1,\;x^2=y^{2^{n-2}},\;xyx^{-1}=y^{-1}\ra.$$

The union $$Q_{2^\infty}=\bigcup_{n\in\IN}Q_{2^n}$$ will be called the {\em infinite group of generalized quaternions}. The quasicyclic group $C_{2^{\infty}}$ has index 2 in $Q_{2^\infty}$ and each element $x\in Q_{2^\infty}\setminus C_{2^\infty}$ has order 4.

The main result of this paper is the following extension of Theorem~5.3.6 \cite{Rob}. It will be essentially used in \cite{BG} for describing of the structure of minimal left ideals of the superextensions of twinic groups.

\begin{theorem}\label{t1} Each 2-group with a unique 2-element subgroup is isomorphic to $C_{2^n}$ or $Q_{2^n}$ for some $n\in\IN\cup\{\infty\}$.
\end{theorem}

As we already know, for finite groups this theorem was proved in \cite[5.3.6]{Rob}. Let us write this fact as a lemma for the future reference: 

\begin{lemma}\label{l1} Each finite 2-group with a unique 2-element subgroup is isomorphic to $C_{2^n}$ or $Q_{2^n}$ for some $n\in\IN$.
\end{lemma}

So, it remains to prove Theorem~\ref{t1} for infinite groups.  The abelian case is easy:

\begin{lemma}\label{l2} Each infinite abelian 2-group $G$ with a unique 2-element subgroup is isomorphic to the quasicyclic 2-group $C_{2^\infty}$.
\end{lemma}

\begin{proof} Let $Z$ be the unique 2-element subgroup of $G$ and $f:Z\to C_2$ be an isomorphism. Since the group $C_{2^\infty}$ is injective, by Baer's Theorem \cite[4.1.2]{Rob}, the homomorphism $f:Z\to C_2\subset C_{2^\infty}$ extends to a homomorphism $\bar f:G\to C_{2^\infty}$. We claim that $\bar f$ is an isomorphism. Indeed, the kernel $\bar f^{-1}(1)$ of $\bar f$ is trivial since it is a 2-group and contains no element of order 2. So, $\bar f$ is inejective and then $\bar f(G)$ concides with $C_{2^\infty}$, being an infinite subgroup of $C_{2^\infty}$.
\end{proof}

The non-abelian case is a bit more difficult. For two elements $a,b$ of a group $G$ by $\la a,b\ra$ we shall denote the subgroup of $G$ generated by the elements $a$ and $b$.  The following lemma gives conditions under which the subgroup $\la a,b\ra$ is finite.

\begin{lemma}\label{l3}  The subgroup $\la a,b\ra$ generated by elements $a,b$ of a group $G$ is finite provided that the following conditions are satisfied:
\begin{enumerate}
\item $b^2\in\la a\ra$;
\item $a^2b\in b\cdot \la a\ra$;
\item the elements $a$ and $ab$ have finite order.
\end{enumerate}
\end{lemma}

\begin{proof} Since the element $a$ has finite order and $b^2\in\la a\ra$, the element $b$ has finite order too. It is clear that the subgroup $H=\la a,b\ra$ generated by the elements $a,b$ can be written as the countable union $H=\bigcup_{k\in\w}H_k$ where $H_0=\{1\}$ and 
$H_k$ is the subset of elements of the form $a^{n_1}b^{m_1}\cdots a^{n_k}b^{m_k}$ where $n_i,m_i\ge 0$ for $i\le k$.

For every $k\in\w$ consider the set 
$$\Pi_k=\{(ab)^ia^j,b(ab)^ia^j:0\le i\le k,\;j\ge 0\}$$ and observe that $\Pi_k\cdot a=\Pi_k$ and $\{1,a,ab\}\cdot \Pi_{k}\subset \Pi_{k+1}$.

\begin{claim} $H_k\subset \Pi_k$ for each $k\in\w$.
\end{claim}

This claim will be proved by induction on $k$. The inclusion $H_0=\{1\}\subset\Pi_0$ is trivial. Assume that for some number $k>0$ the inclusion $H_{k-1}\subset \Pi_{k-1}$ has been proved. 

In order to show that $H_k\subset\Pi_k$, take any element 
$x=a^{n_1}b^{m_1}a^{n_2}b^{m_2}\dots a^{n_k}b^{m_k}\in H_k$. Since $b^2\in\la a\ra$, we can assume that $m_1\in\{0,1\}$. Observe that the product $y=a^{n_2}b^{m_2}\cdots a^{n_k}b^{m_k}\in H_{k-1}$.

If $m_1=0$, then $x=a^{n_1+n_2}j^{m_2}\cdots a^{n_k}j^{m_k}\in H_{k-1}\subset\Pi_{k-1}\subset \Pi_k$.  

Next, assume that $m_1=1$. It follows from $a^2b\in b\cdot\la a\ra$ that for every $n\in\w$ we get $a^{2n}b\in b\cdot \la a\ra$. Write the number $n_1$ 
as $n_1=2n+\varepsilon$ for some $n\in\w$ and some $\varepsilon\in\{0,1\}$. Then $a^{n_1}b=a^\varepsilon a^{2n}b=a^\varepsilon ba^m$ for some $m\in\w$ and hence 
$$x=a^{2n+\varepsilon}b y=a^{\varepsilon}ba^my\in a^\varepsilon b a^m\cdot H_{k-1}=a^\e b\cdot H_{k-1}\subset a^\e b\cdot \Pi_{k-1}\subset \Pi_k.$$
This completes the proof of the claim.

Since the element $ab$ has finite order, we see that the union $\bigcup_{k\in\w}\Pi_k$ is finite and so is the subgroup $H=\bigcup_{k\in\w}\subset\bigcup_{k\in\w}\Pi_k$. 
\end{proof}

The proof of Theorem~\ref{t1} will be complete as soon as we prove that each infinite non-abelian 2-group $G$ with a unique element of order 2 is isomorphic to $Q_{2^\infty}$. Let $1$ denote the neutral element of $G$ and $-1$ denote the unique element of order 2 in $G$. It commutes with any other element of $G$. 

Now we prove a series of lemmas and in the final Lemma~\ref{l12} we shall prove that $G$ is isomorphic to $Q_{2^\infty}$.

\begin{lemma}\label{l4} The group $G$ contains an element of order 8.
\end{lemma}

\begin{proof} In the opposite case $x^4=1$ for each element $x\in G$. The subgroup $Z=\{1,-1\}$ lies in the center of the group $G$ and hence is normal. Since $x^2\in Z$ for all $x\in G$, the quotient group $G/Z$ is Boolean in the sense that $y^2=1$ for all $y\in G/Z$. Being Boolean, the group $G/Z$ is abelian and locally finite (the latter means that each finite subset of $G/Z$ generates a finite subgroup). Then the group $G$ is locally finite too. Since $G$ is infinite, it contains a finite subgroup $H$ of order $|H|\ge 16$. By Lemma~\ref{l1}, $H$ is isomorphic to $C_{2^n}$ or $Q_{2^n}$ for some $n\ge 4$. In both cases $H$ contains an element of order 8, which contradicts our hypothesis.
\end{proof}

Let $F=\{x\in G:\;x^2=-1\}$ denote the set of elements of order 4 in the group $G$.

Lemmas~\ref{l1} and \ref{l2} imply:

\begin{lemma}\label{l5} $|F\cap A|\le 2$ for each abelian subgroup $A\subset G$.
\end{lemma}

\begin{lemma}\label{l6} For each $x\in G$ and  $b\in F\setminus\la x\ra$ we get $bxb^{-1}=x^{-1}$. 
\end{lemma}

\begin{proof} This lemma will be proved by induction on the order $2^k$ of the element $x$. The equality $bxb^{-1}=x^{-1}$ is true if $x$ has oder $\le 2$ (in which case $x$ is equal to $1$ or $-1$). 

Next, we check that the lemma is true if $k=2$. In this case $b^2=x^2=-1\in\la x\ra$ and $xb^2=xx^2=x^2x=b^2x\in b^2\cdot\la x\ra$. By Lemma~\ref{l3}, the subgroup $\la x,b\ra$ is finite. Now we see that  $\la x,b\ra$ is a finite 2-group with a single element of order 2, $\la x,b\ra$ is generated two elements of order 4 and contains two distinct cyclic subgroups of order 4. Lemma~\ref{l1} implies that $Q_8$ is a unique group with these properties. Analyzing the structure of the quaternion group $Q_8$, we see that $bxb^{-1}=x^{-1}$ (because $b$ and $x$ generate two distinct cyclic subgroups of order 4).
 
Now assume that for some $n\ge 3$ we have proved that $bxb^{-1}=x^{-1}$ for any element $x\in G$ of order $2^k<2^n$ such that $b\in F\setminus\la x\ra$. Let $x\in G$ be an element of order $2^n$ and $b\in F\setminus\la x\ra$.
Then the element $x^{-2}$ has order $2^{n-1}\ge4$ and $b\in F\setminus\la x^{-2}\ra$. By the inductive hypothesis, $bx^{-2}b^{-1}=x^{2}$, which implies $x^2b=bx^{-2}\in b\cdot\la x\ra$.  By Lemma~\ref{l3}, the subgroup $\la x,b\ra$ is finite. Since $bx^{-2}=x^2b\ne x^{-2}b$, the subgroup $\la x,b\ra$ is not abelian and by Lemma~\ref{l1}, it is isomorphic  to $Q_{2^m}$ for some $m$. Now the properties of the group $Q_{2^m}$ imply that $bxb^{-1}=x^{-1}$. 
\end{proof}

\begin{lemma}\label{l7} For each maximal abelian subgroup $A\subset G$ of cardinality $|A|>4$ and each $b\in F\setminus A$, we get $F\setminus A=bA$.
\end{lemma}

\begin{proof} By Lemmas~\ref{l1} and \ref{l2}, the group $A$ is isomorphic to $C_{2^m}$ for some $3\le m\le\infty$. Take any element $b\in F\setminus A$. To see that $bA\subset F\setminus A$, take any element $x\in A$. The inclusion $bx\in F\setminus A$ is trivial if $x$ has order $\le 2$. So we assume that $x$ has order $\ge 4$. Since $b\notin A$, we see that $b\notin\la x\ra$. By Lemma~\ref{l6}, $bxb^{-1}=x^{-1}$. Then $bxbx=bxb^{-1}b^2x=x^{-1}(-1)x=-1$, which means that $bx\in F$. Since $x\in A$ and $b\notin A$, we get $bx\in G\setminus A$. 
Thus $bA\subset F\setminus A$.

To see that $F\setminus A\subset bA$, take any element $c\in F\setminus A$. By Lemma~\ref{l6}, $cxc^{-1}=x^{-1}$ for all $x\in A$. Then for each $x\in A$, $b^{-1}cxc^{-1}b=b^{-1}x^{-1}b=x$, which means that the element $b^{-1}c$ commutes with all elements of $A$, and thus $b^{-1}c\in A$ by the maximality of $A$. Then $c=b(b^{-1}c)\in bA$.
\end{proof}

\begin{lemma}\label{l8} For each maximal abelian subgroup $A\subset G$ and each $x\in G\setminus A$ with $x^2\in A$ we get $x^2=-1$.
\end{lemma}

\begin{proof} Assuming that $x^2\ne -1$, we conclude that the element  $x^2$ has oder $\ge 4$.
The maximality of $A\not\ni x$ guarantees that $A\ne\la x^2\ra$. By Lemmas~\ref{l1} or \ref{l2}, $A$ is cyclic or quasicyclic, which allows us to find an element $a\in A$ with $a^2=x^2$. Observe that $x^2=a^2\in\la a\ra$ and $ax^2=x^2a\in x^2\cdot\la a\ra$. By Lemma~\ref{l3}, the subgroup $\la a,x\ra$ is finite and by Lemma~\ref{l1}, it is isomorphic to $C_{2^n}$ or $Q_{2^n}$  for some $n\in\IN$. Observe that $\la a\ra$ and $\la x\ra$ are two distinct cyclic subgroups of order $\ge 8$, which cannot happen in the groups $C_{2^n}$ and $Q_{2^n}$. This contradiction completes the proof of the equality $x^2=-1$.
\end{proof}

\begin{lemma}\label{l9} $|F|\ge 10$.
\end{lemma}

\begin{proof} By Lemma~\ref{l4}, the group $G$ contains an element $a$ of order 8. By Zorn's Lemma, the element $a$ lies in some maximal abelian subgroup $A\subset G$. Since $G$ is non-commutative, there is an element $b\in G\setminus A$. Replacing $b$ by a suitable power $b^{2^k}$, we can additionally assume that $b^2\in A$. By Lemma~\ref{l8}, $b^2=-1$ and thus 
 $b\in F\setminus A$.  By Lemma~\ref{l7}, we get $bA=F\setminus A$, which implies that $|F|=|F\cap A|+|bA|\ge 2+8=10$.
\end{proof}

\begin{lemma}\label{l10} For any $n\ge 3$ the group $G$ contains at most one cyclic subgroup of order $2^n$.
\end{lemma}

\begin{proof} Assume that $a,b$ be two elements generating distinct cyclic subgroups of order $2^n$. First we show that these elements do not commute. Otherwise, the subgroup $\la a,b\ra$ is abelian and by Lemma~\ref{l1} is cyclic and hence contains a unique subgroup of order $2^n$. Let $A,B\subset G$ be maximal abelian subgroups  containing the elements $a,b$, respectively.

 Observe that the set 
$$D=(F\cap A)\cup(F\cap B)\cup(F\cap B)a^{-1}$$ contains at most $2\cdot 3=6$ elements. Since $|F|\ge 10$, we can find an element $c\in F\setminus D$. By Lemma~\ref{l7}, $ca\in cA=F\setminus A$ and $cb\in cB=F\setminus B$. The choice of the element $c$ guarantees that $ca\notin F\cap B$ and hence $ca\in (F\setminus A)\cap (F\setminus B)\subset F\setminus B=cB$. Then $a\in B$ and $a$ commutes with $b$, which is a contradiction.
\end{proof}

Let $A\subset G$ be a maximal abelian subgroup of cardinality $\ge 8$. Such a subgroup exists by Zorn's Lemma and Lemma~\ref{l4}.

\begin{lemma}\label{l11} $G\setminus A=F\setminus A$ and $A$ is a normal subgroup of index 2 in $G$.
\end{lemma}

\begin{proof} The inequality $G\setminus A\ne F\setminus A$ implies the existence of an element $x\in G\setminus A$ of order $2^n\ge 8$. Find a number $k<n$ such that $x^{2^{k}}\notin A$ but $x^{2^{k+1}}\in A$. By Lemma~\ref{l8}, $x^{2^{k+1}}=-1$ and thus $k=n-2\ge 1$. Then the element $z=x^{2^{k-1}}$ has order 8 and does not belong to $A$ as $z^2\notin A$.  By Lemma~\ref{l10}, $G$ contains a unique cyclic subgroup of order 8, which is a subgroup of $A$. Consequently, $z\in\la z\ra\subset A$ and this is a contradiction proving the equality $G\setminus A=F\setminus A$.

By Lemma~\ref{l7}, for any $b\in G\setminus A=F\setminus A$ we get $bA=F\setminus A=G\setminus A$, which means that $A$ has index 2 in $G$ and is normal.
\end{proof}

\begin{lemma}\label{l12} The group $G$ is isomorphic to $Q_{2^\infty}$.
\end{lemma}

\begin{proof} The subgroup $A$ is infinite (as a subgroup of finite index in the infinite group $G$). By Lemma~\ref{l2}, there is an isomorphism $\varphi:A\to C_{2^\infty}$. Given any elements $b\in F\setminus A$ and $\bar\varphi(b)\in Q_{2^\infty}\setminus C_{2^\infty}$, extend $\varphi$ to an isomorphism $\bar\varphi:G\to Q_{2^\infty}$ letting $\bar\varphi(bx)=\bar\varphi(b)\varphi(x)$ for $x\in A$. Using Lemma~\ref{l6} it is easy to check that $\bar\varphi:G\to Q_{2^\infty}$ is a well-defined isomorphism between the groups $G$ and $Q_{2^\infty}$.
\end{proof}

\section{Acknowledgments}

The anonumous referee of this paper pointed out that Theorem~\ref{t1} can be deduced from an old result of Shunkov \cite{Shun}.

\end{document}